\documentclass{amsart}
\usepackage{amsxtra,amssymb,amsmath,amsthm,amsbsy}
\usepackage{graphicx}
\usepackage[all]{xy}

\usepackage{hyperref}
\usepackage{enumerate}

\theoremstyle{plain}
\newtheorem{theorem}{Theorem}[section]
\newtheorem{lemma}[theorem]{Lemma}
\newtheorem{proposition}[theorem]{Proposition}
\newtheorem{corollary}[theorem]{Corollary}

\theoremstyle{definition}

\newtheorem{example}[theorem]{Example}
\newtheorem{remark}[theorem]{Remark}

\newcommand\flowsfrom{\mathrel{\reflectbox{$\leadsto$}}}

\def\dashto{\dashrightarrow}

\def\G{\mathcal G}
\def\H{\mathcal H}
\def\X{\mathcal X}
\def\U{\mathcal U}

\def\R{\mathbb R}
\def\Z{\mathbb Z}
\def\C{C^\infty}

\def\id{{\rm id}}

\def\deck{{\rm Deck}}
\def\diff{{\rm Diff}}

\def\then{\Rightarrow}
\def\from{\leftarrow}
\def\xto{\xrightarrow}
\def\xfrom{\xleftarrow}
\def\action{\curvearrowright}
\def\toto{\rightrightarrows}

\def\tensor{\otimes}

\def\<{\langle}
\def\>{\rangle}

\def\tf{\tilde{f}}
\def\tg{\tilde{g}}

\begin{document}

\title[]
{Discrete dynamics and differentiable stacks}


\author[]
{Alejandro Cabrera \and Matias del Hoyo \and Enrique Pujals}

\address{Departamento de Matem\'atica Aplicada - IM, Universidade
Federal do Rio de Janeiro, CEP 21941-909, Rio de Janeiro, Brazil}
\email{alejandro@matematica.ufrj.br}

\address{Departamento de Geometria - IME, Universidade Federal Fluminense. Rua Professor Marcos Waldemar de Freitas Reis, s/n, 24210-201, Niter\'oi, Brazil}
\email{mldelhoyo@id.uff.br}

\address{Instituto Nacional de Matematica Pura e Aplicada, Estrada Dona Castorina 110, Rio de Janeiro, 22460-320, Brazil} 
\email{enrique@impa.br}

\date{\today}


\begin{abstract}
In this paper we relate the study of actions of discrete groups over connected manifolds to that of their orbit spaces seen as differentiable stacks. We show that the orbit stack of a discrete dynamical system on a simply connected manifold encodes the dynamics up to conjugation and inversion. We also prove a generalization of this result for arbitrary discrete groups and non-simply connected manifolds, and relate it to the covering theory of stacks. As applications, we obtain a geometric version of Rieffel's theorem on irrational rotations of the circle, we compute the stack-theoretic fundamental group of hyperbolic toral automorphisms, and we revisit the classification of lens spaces.
\end{abstract}


\maketitle

\setcounter{tocdepth}{1} 
\tableofcontents





\section{Introduction}


A {\em discrete dynamical system} consists of a {\em state space} $M$ and an {\em evolution rule} $f:M\to M$ according to which points evolve by iterations $f^n:M\to M$, $n\in\mathbb Z$. The main goal of the theory is to classify $f$ up to conjugacy $f\mapsto hfh^{-1}$ by automorphisms $h$ of $M$. We will restrict ourselves to the smooth setting, so that $M$ is a manifold, $f$ is a diffeomorphism and all maps are of class $C^ \infty$.
Basic invariants of a dynamical system under conjugation and inversion are the subspaces of {\em fixed points} and {\em periodic points} ${\rm Fix}_n=\{x|f^n(x)=x\}\subset M$, and the {\em orbit space} $M/f=\{[x]|x\sim f(x)\}$. We begin this introduction discussing up to which extent the orbit space encodes the dynamics.


In very special cases, namely when $M$ is simply connected and the induced action $\Z\action_f M$ is free and proper, the orbit space is a manifold which allows us to recover the dynamics up to conjugation and inversion, by means of covering space theory. As an example, think of the translation $\tau: x \mapsto x+1$ on $M=\R$, the orbit space is the circle $S^1$ and we can reconstruct the dynamics as the natural action of $\pi_1(S^1)=\Z$ over the universal cover $\R$. This picture can be generalized, we can relax the simply connectedness of $M$ to requiring $M$ connected and $H^1(M)=0$, and arrive to the following statement:

\smallskip

{\em If $M$ is connected, $H^1(M)=0$ and $\Z\action_f M$ is free and proper, then the manifold $M/f$ recovers the dynamics $f$ up to conjugation and inversion.}

\smallskip

This statement works as follows. The free and proper action $\Z\action_f M$ corresponds to a smooth principal $\Z$-bundle $M\to M/f$ (see eg. \cite[Ex. 3.2.4]{dh}), so $M\to M/f$ is a covering space with fiber $\Z$, and we get a short exact sequence
$$1\to \pi_1(M) \to \pi_1(M/f)\xto{c_f}\Z\to 1$$
Since $H^1(M)={\rm Hom}(\pi_1(M),\Z)=0$, any epimorphism $\pi_1(M/f)\to\Z$ must vanish over $\pi_1(M)$, and $c_f$ is completely determined by $M/f$ up to sign. We can then recover $M$ as the quotient of the universal covering space \smash{$\widetilde{M/f}$} by the action of $\ker c_f\subset \pi_1(M/f)$, and also recover the action $\Z\action_f M$ as the deck transformations of $M\to M/f$.

Let us now think of more general situations. In the previous discussion, the condition $H^1(M)=0$ can be relaxed, as long as we keep track of the cohomology class $c_f\in H^1(M/f)$, as we will do later. But getting rid of the freeness and properness of the dynamics is more delicate. For non-free actions, the presence of fixed and periodic points may not be reflected on the space $M/f$, as in the case of rational rotations on the circle. For non-proper actions, the topology on the quotient $M/f$ can be quite pathological and fail to be a manifold, as in the case of irrational rotations on the circle. 


The aim of this paper is to show that the above statement has a generalization for arbitrary discrete dynamics if we think of the space of orbits not as a mere topological space but as a {\em differentiable stack}. 
Stacks are enhanced spaces with virtual symmetries attached, and differentiable stacks are those arising as quotients of manifolds. Examples include manifolds and orbifolds, but can be much more general, have non-Hausdorff topologies, or admit isotropies of positive dimension. As we shall see, this notion exactly incorporates the relevant information underlying the orbits of a dynamical system up to conjugation and inversion.


The classical definition of stacks involves sophisticated category theory and the technical paraphernalia often strays from the original geometric intuition. We follow here an alternative approach by using {\em Lie groupoids}.
A groupoid consists of a set of objects and arrows between them that can be composed satisfying the usual group-like axioms: associativity, existence of units and inverses.
A {\em Lie} groupoid is a groupoid on which both objects and arrows are differentiable manifolds, and the structural maps are smooth. They provide a unified framework to deal with various classical geometries such as group actions, foliations and principal bundles, codifying internal and external symmetries. 


Just as ordinary quotients arise from equivalence relations, differentiable stacks can be thought of as quotients by the enhanced equivalence relation defined by a Lie groupoid. Two different Lie groupoids can {\em present} the same stack, and this happens exactly when they are {\em Morita equivalent}, so that working with differentiable stacks up to isomorphism is equivalent to working with Lie groupoids up to Morita equivalences. A key ingredient in our paper is a concrete characterization of Morita equivalences in terms of {\em Morita morphisms}, introduced in \cite{dh}, and equivalent to a more standard yet abstract approach using bimodules.

\smallskip


In this paper we then relate the study of discrete dynamical systems, and more generally, actions of discrete groups over connected manifolds, to the study of the underlying orbit stack defined by the corresponding {\em action groupoids}: 
$$
\begin{matrix}\text{dynamics}\\ f:M\to M \end{matrix}\quad \to\quad
\begin{matrix}\text{actions}\\ \Z \action_f M \end{matrix}\quad \to\quad
\begin{matrix}\text{groupoids}\\ \Z \ltimes_f M\toto M\end{matrix}\quad \to\quad
\begin{matrix}\text{stacks}\\ M//f \end{matrix}
$$
 
First we show that because of our topological assumptions, every morphism between our action groupoids must decompose as a product of a group homomorphism and a smooth map. Then we review the lifting of a dynamics to the universal covering space, and derive from this construction examples of Morita morphisms. 
Our first main result is the Characterization Theorem \ref{thm1}, which essentially says that every Morita morphism between our action groupoids must arise from a covering space as before. Our second main result is the Representation Theorem \ref{thm2}, implying that any Morita equivalence between action groupoids can be determined by a groupoid map from the universal cover lift. As an immediate corollary, we have:

\smallskip

{\em If $M$ is connected and $H^1(M)=0$ then the differentiable stack $M//f$ recovers the dynamics $f$ up to conjugation and inversion.}

\smallskip

In the general case of an action $G\action M$ of a discrete group on a connected manifold, we see that the lifting to the universal covering space is a complete invariant of the stack $M//G$.
From the theoretic viewpoint, we explain how to interpret the lifted group to the universal cover as a stacky theoretic fundamental group, and comment on the covering map theory of differentiable stacks, studied for instance in \cite{n}.
We also show how to completely characterize a dynamics $f:M\to M$ from the stacky viewpoint without the cohomological assumption, by attaching to the orbit stack $M//f$  a characteristic class $c_f\in H^1(M//f)$.

We include application of our results to concrete simple examples, to illustrate the power of our results, and also to keep the paper readable for a broad audience, specially dynamicists and geometers that are not familiarized with the notion of stack. We derive a classification by Morita classes of the irrational rotations of the circle, which is a geometric version of a celebrated theorem by Rieffel \cite{r} in the context of $C^\ast$-algebras. We also classify hyperbolic toral automorphisms by computing their stacky fundamental group. Finally we revisit the classification of lens spaces, illustrating that even when the orbit space of a dynamics is a manifold, its orbit stack retains extra information.

\smallskip

\noindent{\bf Organization.}
We give a self-contained introduction to Lie groupoids in section 2, discuss the groupoids arising from actions of discrete groups on connected manifolds on section 3, and review the lifting of a dynamics to the universal covering space in section 4.
Section 5 and 6 are devoted to our two fundamental theorems, the characterization of the Morita morphisms, and the representation of fractions using the universal cover, respectively. We link our work with the general theory of coverings of stacks in section 7, providing independent proofs, and commenting on the existing bibliography. Then in section 8 we introduced the characteristic class of a dynamics. The last section 9 discusses four concrete examples: irrational rotations on the circle, hyperbolic toral automorphism, periodic dynamics and orbifolds, and lens spaces.

\smallskip

\noindent{\bf Acknowledgements.} 
We thank H. Bursztyn, P. Carrasco, R. Exel and M. Shub for several conversations during the development of this project, IMPA for providing an excelent environment for our meetings, and CNPq and FAPERJ for finantial support.

\section{Background on Lie groupoids}

We review here the basic definitions about Lie groupoids and discuss the  fundamental examples that will appear in our paper. We refer to \cite{cw,cf,dh,mm} for the proofs of the statements and a more thorough introduction to Lie groupoids.

\medskip


Given a manifold $M$, a {\bf Lie groupoid} $\G\toto M$ consists of the following data: a manifold $\G$ of arrows between 
the points of $M$, two submersions $s,t:\G\to M$ indicating the source and target of each arrow, a partial associative smooth multiplication
$m:\G\times_M\G\to\G$, \smash{$(z\xfrom h y,y\xfrom g x)\mapsto (z\xfrom{hg} x)$}, a unit map $u:M\to\G$, \smash{$x\mapsto (x\xfrom{u_x} x)$}, and an inverse map $i:\G\to\G$, \smash{$(y\xfrom g x)\mapsto (x\xfrom{g^{-1}}y)$}, satisfying the axioms $u_yg=g=gu_x$, $g^{-1}g=u_x$ and $gg^ {-1}=u_y$. We often regard $M$ as an embedded submanifold of $\G$ via $u$.


We can think of a groupoid $\G\toto M$ as an enhanced equivalence relation between points of the manifold $M$, where arrows parameterize different ways in which two given points can be equivalent. Given $x\in M$, its {\bf orbit} $O_x\subset M$ consists of all the equivalent points $y$, i.e. points for which there exists an arrow $y\xfrom g x$ on $\G$.  The orbits are immersed submanifolds and define a (possibly singular) {\bf characteristic foliation} on $M$. The underlying {\bf orbit space} $M/\G$ is given the quotient topology and it may be non-Hausdorff. If there is only one orbit then $\G$ is called {\bf transitive}. As we will see later, the language of differentiable stacks allow us to regard $M/\G$ as a richer geometric object, on which we can perform differential geometry.


Given $x\in M$, its {\bf isotropy} $\G_x=s^{-1}(x)\cap t^{-1}(x)$ is a Lie group, embedded into $\G$. If $x$ and $y$ are points on the same orbit, then the isotropies $\G_x$ and $\G_y$ turn out to be isomorphic, the isomorphism is given by conjugation by any arrow $y\xfrom g x$, and is canonical if and only if $\G_x$ is abelian. The isotropy group $\G_x$ naturally acts on the {\bf normal vector space} $\G_x\action N_xO=T_xM/T_xO_x$ via  $g[\gamma'(0)]=[(t\tilde\gamma)'(0)]$, where $\gamma$ is a curve with $\gamma(0)=x$ and $\tilde\gamma$ is a local lift along $s:\G \to M$ passing through $g$.


A {\bf morphism} of Lie groupoids $\phi:(\G\toto M)\to (\H\toto N)$ consists of a smooth map $\phi:\G\to\H$ which restricts to $\varphi:M\to N$ via the units, and preserves the groupoid structure: source, target, multiplication, unit and inverse. Such a $\phi$ induces a continuous map on the orbit spaces $\bar\varphi:M/\G\to N/\H$, and morphisms on the isotropy $\phi_x:\G_x\to\H_{\varphi(x)}$, and the normal directions $d\varphi_x:N_xO\to N_{\varphi(x)}O$. Two morphisms $\phi,\phi'$ are {\bf isomorphic} if there is a map $\theta:M\to\H$ such that $s(\theta_x)=\varphi(x)$, $t(\theta_x)=\varphi'(x)$, and such that $\phi'(g)\theta_x=\theta_y\phi(g)$ for all \smash{$y\xfrom g x$}.


\begin{example}
We can see manifolds and Lie groups as examples of Lie groupoids. To a manifold $M$ we associate the {\bf unit groupoid} $M\toto M$ with only identity arrows, it represents the trivial dynamics, the orbits have a single point and the isotropy is trivial. To a Lie group $G$ we associate the groupoid with a single object $G\toto\ast$, we see elements of $G$ as symmetries of a single formal object. Morphisms between unit groupoids are the same as maps between the manifolds, and similarly for Lie groups.
\end{example}

\begin{example}\label{ex:action}
Given $G\action M$ an action of a Lie group $G$ over a manifold $M$, the {\bf action groupoid} $G\ltimes M=(G\times M\toto M)$
has source the projection, target the action map $\rho:G\times M\to M$, and multiplication given by that of $G$.
As a particular case, given $f:M\to M$ a diffeomorphism, its iterations and inverses define an action by the integers $\Z\action_f M$, $n\cdot x=f^n(x)$. The corresponding action groupoid will be denoted by $\Z\ltimes_f M$. We will discuss morphisms of these groupoids in the next section.
\end{example}

\begin{example}\label{ex:relation}
Let $M$ be a manifold and let $R\subset M\times M$ be an equivalence relation that is a closed embedded submanifold and such that the projections $\pi_i|_R:R\to M$ are submersive. Then we can define a {\bf relation groupoid} $R\toto M$, with exactly one arrow between two points if they belong to the same class. It follows by the {\em Godement criterion} that $M/R$ is a manifold and $\pi:M\to M/R$ a surjective submersion. In this way we get a correspondence between relation groupoids and surjective submersions $M\to N$, and between morphisms and maps that descend to the quotients.
\end{example}

\begin{example}
Two particular cases of previous example that are interesting on their own. Given $M$ a manifold, the {\bf pair groupoid} $M\times M\toto M$ has exactly one arrow between any given two objects, it corresponds to the trivial submersion $M\to\ast$. If $\U=\{U_i\}$ is an open cover of $M$, the {\bf Cech groupoid} $\U=(\coprod_{ij}U_{ij}\toto\coprod_{i}U_i)$ is the relation groupoid arising from the surjection $\coprod_i U_i\to M$ given by inclusion. This is the domain of definion of the usual Cech cocycles arising from local trivialization of principal bundles -- they are the same as groupoid morphisms $\U\to G$.
\end{example}

\begin{example}\label{ex:gauge}
Given $G$ a Lie group and $\pi:P\to M$ a principal $G$-bundle, its {\bf gauge groupoid} $P\tensor^G P\toto M$ is the quotient of the pair groupoid of $P$ by the diagonal action $P\times P\curvearrowleft G$. We can think of $P\tensor^G P$ as having the fibers of $\pi$ as objects, and the $G$-equivariant maps between them as arrows. It is transitive and the isotropy identifies with $G$. This construction sets a 1-1 correspodence between transitive groupoids and principal bundles. 
\end{example}

Other important examples of Lie groupoids arise from foliation theory and pseudogroups. In these cases it is better to work with manifolds $\G$ that may not be Hausdorff, so as to include every example. Since we are not dealing with these constructions we will assume that all our manifolds are Hausdorff.

\section{Discrete dynamics as groupoids}

We focus now on the Lie groupoids arising from discrete dynamics, discuss morphisms at the level of dynamics, actions and groupoids, relate them to each other, and establish a key lemma relying on our topological assumptions.

\medskip


Given two diffeomorphisms $f:M\to M$ and $f':M'\to M'$, we define a {\bf morphism of dynamics} $\phi:f\then f'$ as a smooth map $\phi:M\to M'$ such that $ f' \circ \phi = \phi \circ f$. In this way, $f$ and $f'$ are {\bf conjugated} by $\phi$, namely $f'= \phi \circ f \circ \phi^{-1}$, if they are related by an invertible morphism. We will often set a base manifold $M$ and study different dynamics on it, so $f,f'$ are conjugated if they are so as elements of the group of diffeomorphism $\diff(M)$.


Given Lie group actions $G\action M$ and $H\action N$, we define a {\bf morphism of actions} $\phi:(G\action M)\to (H\action N)$ as a pair $(\lambda,\phi)$ where $\lambda:G\to H$ is a group morphism and $\phi:M\to N$ is a smooth map satisfying $\phi(gx)=\lambda(g)\phi(x)$ for all $g,x$. Note that if two actions are isomorphic then necessarily the groups are isomorphic and the manifolds are diffeomorphic. It is rather clear that the construction
$$\text{dynamics}\to\text{actions} \qquad (f:M\to M) \mapsto (\Z\action_fM)$$
is {\bf functorial}, in the sense that a morphism of dynamics always induces a morphism of actions. But there are more morphisms between the induced actions than between the original dynamics. In fact, since $\id$ and $-\id$ are the only automorphisms of $\Z$, two actions $\Z\action_fM$ and $\Z\action_{f'}M$ are isomorphic if and only if $f'$ is conjugated to either $f$ or its inverse $f^{-1}$.


We are concerned now with the action groupoid construction reviewed in example \ref{ex:action}: 
$$\text{actions}\to\text{groupoids} \qquad (G\action M) \mapsto (G\ltimes M\toto M)$$
The next simple examples show that, when working with $M$ disconnected or $G$ of positive dimension, the passing from actions to groupoids might involve a sensible loss of information.

\begin{example}
Two Lie groups $G,H$ that are diffeomorphic but not isomorphic, when acting on themselves by left multiplication, yield the same Lie groupoid, namely the pair groupoid of the underlying manifold. For instance, we can take $G,H$ to be non-isomorphic finite groups of the same order, or more geometrically, we can take the plane $\R^2$ with the multiplication $(x,y)\cdot_\lambda(x',y')=(x+x',y+e^{\lambda x}y')$, that is commutative for $\lambda=0$ and non-commutative for $\lambda\neq 0$.
\end{example}

However, when restricting ourselves to actions of discrete groups on connected manifolds, the following lemma holds, playing a key role in our passage from dynamics to groupoids. In a categorical language, it sates that the action groupoid construction, under our topological assumptions, is fully faithful.

\begin{lemma}\label{lemma:key}
Let $G\action M$, $H\action N$ be discrete groups acting on connected manifolds.
Any Lie groupoid map $\phi:G\ltimes M\to H\ltimes N$ has the form $\phi(g,x)=(\lambda(g),\varphi(x))$ with $\lambda:G\to H$ a group morphism and $\varphi:M\to N$ a smooth map.
\end{lemma}

\begin{proof}
Starting with $\phi:G\ltimes M \to H\ltimes N$ a groupoid morphism,
let $\varphi:M\to N$ be the map at the level of objects, and write $\phi=(\phi_1,\phi_2):G\times M\to H\times N$ for the components. Since $\phi$ preserves the source map, which is just the second projection, we have $\phi_2(n,x)=\varphi(x)$. And since $M$ is connected, $\phi_1(g,x)=\phi_1(g,y)$ for all $x,y$, and we gain a map $\lambda:G\to H$ by defining $\lambda(g)=\phi_1(g\times M)$. Picking an arbitrary $x\in M$, since $\phi$ is a groupoid morphism, we have the equation
$$(\lambda(e_G),\varphi(x))=\phi(e_G,x)=\phi(u_x)=u_{\varphi(x)}=(e_H,\varphi(x))$$
and for every $g,h\in G$ we have
\begin{multline*}
(\lambda(gh),\varphi(x))=\phi(gh,x)=\phi((g,hx)(h,x))=\\
=\phi(g,hx)\phi(h,x)=(\lambda(g),\varphi(hx))(\lambda(h),\varphi(x))=(\lambda(g)\lambda(h),\varphi(x)).
\end{multline*}
It follows that $\lambda$ preserves the neutral element and the multiplication, and it is therefore a group morphism. 
\end{proof}


The last step in our construction will be to pass from Lie groupoids to differentiable stacks:
$$\text{groupoids}\to\text{stacks} \qquad (\G\toto M) \mapsto M//\G$$
We can think of $M//\G$ as the topological space $M/\G$ endowed with some further smooth information, that also keeps track of the isotropy groups. 
For a formal definition, we will obtain differentiable stacks from Lie groupoids by inverting a special class of groupoid maps, called Morita maps. We will discuss these notions in detail later.

\begin{remark}
There is an alternative point of view into working with possibly singular quotients such as $M/\G$ which is more analytic in nature. The idea is to consider (typically non-commutative) $C^\ast$-algebras as providing a model for smooth functions on $M/\G$ (see \cite{con,r}). The two approaches can be related at a general level via the convolution algebra construction $\G \mapsto C^\ast(\G)$, see \cite{l}. In this paper we focus on the purely geometric point of view, dealing directly with stacks as spaces instead of working with the 'non-commutative algebras of functions' underlying them.
\end{remark}

\section{Lifting to the universal cover}
\label{sec:lifting}

Before moving into Morita maps and differentiable stacks, let us overview some basic constructions from topology yielding a lift of an action of a discrete group over a manifold to its universal cover. The resulting lifted action will play a central role in our characterization of orbit stacks $M//G$ when $M$ is not simply connected.

\medskip


Given $M$ a connected manifold and given $p:\tilde M\to M$ its universal covering map, recall that a {\bf deck transformation} is a diffeomorphism $\phi:\tilde M\to\tilde M$ lifting the identity on $M$: $p\phi=p$. Deck transformations form a group $\deck(\tilde M\to M)\subset\diff(\tilde M)$ which is isomorphic to the fundamental group of $M$, this is an instance of the following more general fact. Fix $\tilde x\in\tilde M$, call $x=p(\tilde x)$, and let $\phi:\tilde M\to \tilde M$ be a diffeomorphism. We denote by $\alpha_\phi$ the homotopy class of the projected path $p_*(\tilde x\flowsfrom\phi(\tilde x))$ in $M$.

\begin{lemma}\label{lemma:deck}
The association $\phi\mapsto\alpha_\phi$ gives a 1-1 correspondence between liftings $\phi$ of a given $\varphi:M \to M$ to $\tilde M$, $p\phi = \varphi p$, and homotopy classes of paths $x\flowsfrom \varphi(x)$. When $\varphi=\id$ this is a group isomorphism between the deck transformations and the fundamental group.
\end{lemma}

\begin{proof}
We can identify the universal cover with the homotopy classes of paths ending at $x=p(\tilde x)$,
$$\tilde M\xto\sim P_xM \qquad \tilde y \mapsto p_*(\tilde x\flowsfrom\tilde y).$$
Under this identification, $p$ agrees with taking the path's starting point.
Given \smash{$x\overset\alpha\flowsfrom \varphi(x)$}, we define $\phi_\alpha:P_xM\to P_xM$ by concatenation 
$$\phi_\alpha(x\overset\gamma\flowsfrom y)=x\overset\alpha\flowsfrom \varphi(x)\overset{\varphi(\gamma)}\flowsfrom\varphi(y),$$ which clearly covers $\varphi$. This gives the inverse of $\phi \mapsto \alpha_\phi$, and it obviously preserves the group structure for $\varphi=\id$.
\end{proof}


Given $\rho:G\action M$ a discrete group acting on a connected manifold, 
let $\tilde G\subset G\times \diff(\tilde M)$ be the subgroup of compatible
pairs $(g,\phi)$ in the sense that $p\circ \phi = \rho(g)\circ p$.
The group $\tilde G$ naturally acts on $\tilde M$ via the second projection, thus defining a {\bf lifted action} $\tilde\rho:\tilde G \action \tilde M$ and the associated {\bf lifted action groupoid} $ \tilde G \ltimes \tilde M \toto \tilde M$. 


\begin{remark}\label{rmk:group-extension}
By the lifting property of the universal cover, the projection gives a group epimorphism $\pi:\tilde G\to G$ with kernel the group of deck transformations, that we identify with $\pi_1(M)$, hence giving a {\bf canonical group extension}:
$$1 \to \pi_1(M)\to \tilde G\to G \to 1$$
It follows that $\tilde G$ is discrete for any reasonable topology on $\diff(M)$. 
If $\pi:\tilde G\to G$ admits a section, namely when we have a homomorphism lifting the $G$-transformations on $M$ to $\tilde M$, then we have a semi-direct product $\tilde G\cong G\ltimes\pi_1(M)$. 
\end{remark}

\begin{example}\label{ex:lifted-group}
If $G$ is a free group, as it is in the actions arising from a dynamics $f:M\to M$, then the canonical group extension splits.
Given $f:M\to M$ a dynamics, the choice of a particular lift $\tilde f : \tilde M \to \tilde M$ induces an isomorphism of groups 
$$\Z \ltimes_{\tf} \pi_1(M)\xto\cong \tilde G 
\qquad
(n, \phi) \mapsto (n,\tilde f^n \phi)$$
where $\phi$ is a deck transformation.
The product on $\tilde G$ is given by
$ (n,\phi)(n',\phi') = (n+n', \tilde f^{-n'}\phi \tilde f^{n'} \phi')$,
while the action on $\tilde M$ is given by
\[(n,\phi)\cdot_{\tilde f} x = \tilde f^n(\phi(x)).\]
In the case when $M=S^1$ and $f:S^1 \to S^1$, $\pi_1(M)$ identifies with the integer translations on $\tilde M=\R$, and choosing any particular lift $\tilde f$, the product and the action are
$$(m,n)(m',n')=(m+m',\epsilon^m n + n')
\qquad
(m,n)\cdot x = \tilde{f}^m(x) + \epsilon^m n
$$
with $\epsilon=\pm 1$ being a constant depending on whether $f$ preserves the orientation.
\end{example}


\begin{example}
If $\rho:G\action M$ has a fixed point $x$, namely $g\cdot x=x$ for all $g\in G$, then the canonical group extension splits. In fact, after chosing $\tilde x\in\tilde M$ over $x\in M$, we can lift each diffeomorphism $\rho_g$ to a diffeomorphism $\tilde\rho_g$ that has $\tilde x$ as a fixed point, and this lifting is unique, hence defining a section $G\to\tilde G$.
One subexample: if $G$ acts trivially on $M$ then $\tilde G\cong G\times\pi_1(M)$ is the direct product of groups and the lifted action is the product of the trivial action with the one by deck transformations. Other subexample: a toral automorphism $f_A:\R^n/\Z^n\to \R^n/\Z^n$ given by a matrix $A\in GL_n(\Z)$ has $\bar 0$ as a fixed point, in this case $\tilde G\cong \Z\ltimes_A\Z^n$.
\end{example}


\begin{example}
An example on which $\tilde G$ is not a semi-direct product. Let $M=P^3(\R)$ and let $G=\Z_2$ acting on $M$ via $\bar 1\cdot (x:y:z:w)=(-y:x:-w:z)$. Then $\tilde G\cong\Z_4$ and the canonical group extension does not split.
\end{example}


The projection $\pi:\tilde G\to G$ together with the covering projection $p:\tilde M\to M$ define a morphism of actions and of their induced groupoids:
$$(\pi,p):(\tilde G\action\tilde M)\to (G\action M), 
\qquad
(\pi,p):\tilde G\ltimes\tilde M \to G\ltimes M.$$
It is easy to check that they induce an homeomorphism between the orbit spaces $\tilde M/\tilde G\cong M/G$. In the next section we will show that this also holds at the richer level of stacks. 



\section{Morita maps between discrete dynamics}


In this section we discuss the notion of Morita morphisms, which is a morphism betwen Lie groupoids inducing an isomorphism between the orbit stacks, and present our first main theorem, a characterization result for such morphisms between our groupoids. We derive some immediate corollaries.

\bigskip


A morphism of Lie groupoids $\phi:(\G\toto M)\to(\H\toto N)$ is {\bf Morita} if it satisfies the following properties:
\begin{itemize}
 \item $\bar\phi:M/\G\to N/\H$ is a homeomorphism between the orbit spaces; 
 \item $\phi_x:\G_x\to\H_{\phi(x)}$ is a group isomorphism betwen the isotropy groups for all $x$; and
 \item $d\varphi_x:N_xO\to N_{\varphi(x)}O$ is a linear isomorphism between the normal directions for all $x$.
\end{itemize}
It easily follows from this definition that isomorphisms are Morita, that composition of Morita morphisms is Morita, and that if $\phi,\phi'$ are isomorphic Lie groupoid morphisms, then one of them is Morita if and only if the other is so. Then, if a morphism admits a quasi-inverse, namely an inverse up to isomorphism, then it also is Morita. It also follows easily that in a commutative triangle of Lie groupoid morphisms, if two of them are Morita, then so does the other, we call this the {\it 2-for-3 property}. Lastly, we remark that a Morita morphism is set-theoretic {\it fully faithful}, namely it induces bijections $\G(y,x)\to \H(\varphi(y),\varphi(x))$, and set-theoretic essentially surjective, namely the image of $\varphi$ meets every orbit of $\H$.


\begin{example}\label{ex:cech-map}
(cf. \ref{ex:relation})
A morphism between unit groupoids is Morita if and only if it is a diffeomorphism between the manifolds. If $R\toto M$ is a relation groupoid, associated to the submersion $q:M\to N$, then the projection $\pi:(R\toto M)\to (N\toto N)$ is Morita, and there is a Morita morphism on the other direction if and only if $q$ admits a global section.
\end{example}

\begin{example}
(cf. \ref{ex:gauge})
A morphism between Lie groups is Morita if and only if it is an isomorphism. If $\G\toto M$ is a transitive Lie groupoid then the inclusion
$\phi:(\G_x\toto x)\to(\G\toto M)$ is Morita for any $x\in M$. Note that this inclusion admits a quasi-inverse if and only if $\G$ is the gauge groupoid of a trivial principal $G$-bundle.
\end{example}


\begin{example}\label{ex:morita-action}
Given $G\action M$ and given $K\subset G$ a normal subgroup such that $K\action M$ is free and proper, it follows that the orbit space $M/K$ is indeed a manifold (see e.g. \cite[Cor. 2.3.4]{dh}), and that the quotient map
$$\pi:(G\ltimes M\toto M) \to (G/K\ltimes M/K\toto M/K)$$
is a Morita map, for it clearly preserves the orbit space, the isotropy groups and the normal directions. 
\end{example}

Our first main theorem characterizes the Morita morphisms between groupoids arising from discrete dynamics. It can be seen as a converse for previous example.


\begin{theorem}[Characterization Theorem]\label{thm1}
Let $G,H$ be discrete groups acting on $M,N$ connected manifolds.
Then $\phi:G\ltimes M\to H\ltimes N$ is Morita if and only if there is $K\subset G$ normal that acts on $M$ freely and properly and $\phi$ factors as $\bar\phi\circ\pi$ with $\pi$ the quotient map and $\bar\phi$ a groupoid isomorphism.
$$\xymatrix{
G \ltimes M \ar[r]^\phi \ar[d]_\pi & H\ltimes N \\ G/K\ltimes M/K \ar[ur]_{\bar\phi}
}$$
\end{theorem}

\begin{proof}
Let us consider a Morita map $\phi:G\ltimes M\to H\ltimes N$ and show that it has to have the claimed form. By lemma \ref{lemma:key} the morphism is of the form $\phi=(\lambda,\varphi)$, with $\lambda:G\to H$ a group morphism and $\varphi:M\to N$ a smooth map. Since $\phi$ is Morita, and since the orbits are 0-dimensional, it follows that the normal directions agree with the tangent spaces, and $d\varphi:T_xM\to T_{\varphi(x)}N$ is an isomorphism for all $x\in M$, that is, $\varphi$ is a local diffeomorphism.

For $K=\ker\lambda$ we claim that the restricted action $K\action M$ is free and proper. To see that the action is free, note that if $g\in K$, $x\in M$ and $gx=x$, then the arrow $(g,x)\in G\ltimes M$ is in the kernel of the isotropy map $(G\ltimes M)_x\to (H\ltimes N)_{\varphi(x)}$. Since $\phi$ is Morita, this has to be an isomorphism, and therefore $g=e_G$. This shows that the action is free, or equivalently, that the map $K\times M\to M\times M$, $(g,x)\mapsto (gx,x)$, is injective. To see that the action is proper we need to show that its image, the relation defined by the action, is closed.

We will show that the $K$-orbits agree with the fibers of $\varphi:M\to N$, so the relation identifies with $(\varphi\times\varphi)^{-1}(\Delta_N)$, the preimage of the diagonal in $N$, that is closed for our manifolds are assumed to be Hausdorff. On the one hand, since $\phi$ preserves the target maps, we have $\varphi(gx)=\lambda(g)\varphi(x)$, and if $g\in K$ then $\lambda(g)=e_H$ and $x$ and $gx$ are in the same fiber, so the orbits are included in the fibers. On the other hand, if $x,x'\in M$ are such that $\varphi(x)=\varphi(x')=y$, since $\phi$ is set-theoretic fully faithful, we can lift the identity arrow $y\xfrom{(e_H,y)}y$ to an arrow $x'\xfrom{(g,x)}x$, from where $g\in K$ and $gx=x'$, so the fibers are included in the orbits too. 


We can now consider the action of $G/K$ on the manifold $M/K$, and as seen in Example \ref{ex:morita-action}, the quotient map $\pi:G\ltimes M\to G/K\ltimes M/K$ is Morita. The map $\varphi$ factors through this quotient as $\bar\varphi\pi$, and the groupoid morphism $\phi$ also factors as $\bar\phi\pi$. Moreover, by 2-for-3, we get that $\bar\phi=(\bar\lambda,\bar\varphi)$ is Morita. We will prove that it is in fact an isomorphism, namely that $\bar\lambda$ is a group isomorphism and $\bar\varphi$ is a diffeomorphism.

A similar argument to the one given above implies that $\bar\varphi$ is a local diffeomorphism, so it suffices to show that it is bijective.
Moreover, since the $\varphi$-fibers agree with the $K$-orbits, the induced map $\bar\varphi$ is injective. To see that $\bar\varphi$ is surjective, or equivalently that $\varphi$ is surjective, it is enough to show that the image is open and closed, for $N$ is connected. The image is open because $\varphi$ is a local diffeomorphism. 

Now suppose $y\in N\setminus\varphi(M)$, there has to exist $x$ such that $\varphi(x)$ is in the same orbit since $\phi$ yields a bijection between the orbits. Then there is an arrow $(h,\varphi(x)):\varphi(x)\to y$, and $h\notin\lambda(G)$, for if $h=\lambda(g)$ for some $g$ then $y=h\varphi(x)=\lambda(g)\varphi(x)=\varphi(gx)$, which is an absurd. 
Similarly, if $h\varphi(M)\cap\varphi(M)$ is non-empty, then there is an arrow \smash{$z'\xfrom{(h,z)}z$} on $H\ltimes N$ with source and target in $\varphi(M)$, and by set-theoretic fully faithfulness, it follows that $(h,z)$ is in the image of $\phi$ and $h=\lambda(g)$, which is an absurd. We conclude that $h\varphi(M)$ is an open containing $y=h\varphi(x)$ not touching the image $\varphi(M)$, so $N\setminus\varphi(M)$ is open as well. 

Finally, to see that $\bar\lambda:G/K\to H$ is an isomorphism, note that it is injective, and that a Morita morphism that is surjective on objects is surjective on arrows as well, as it follows from the set-theoretic fully faithfulness. Then $\bar\lambda$ has to be surjective, and therefore an isomorphism.
\end{proof}


We immediately conclude the following corollaries.

\begin{corollary}
If $N$ is simply connected, then $\phi:G\ltimes M\to H\ltimes N$ is Morita if and only if it is an isomorphism.
\end{corollary}

\begin{proof}
It follows from our characterization theorem that $M\to N$ is a connected covering map and hence it must be a diffeomorphism.
\end{proof}

\begin{corollary}\label{cor:h1}
If $H^1(N)=0$, then a map $\phi:\Z\ltimes_f M\to \Z\ltimes_g N$ is Morita if and only if it is an isomorphism. 
\end{corollary}

\begin{proof}
We use the notations of Theorem \ref{thm1} and set $\phi = (\lambda, \varphi)$. If $K\subset\Z$ is non-trivial then it has to be infinite cyclic. Now, since $K$ is the group of deck transformations of $M\to N$, there is an epimorphism $\pi_1(N)\to K\cong\Z$,  implying that $H^1(N)\cong\hom(\pi_1(N),\Z)\neq 0$, contradicting our hypothesis. We conclude then that $K=0$ and $\lambda=\pm \id$.
\end{proof}

We close this section with an ad hoc example showing the necessity of the connectedness assumption in the Characterization Theorem \ref{thm1}.

\begin{example}
Let $M=\Z\times\Z$ and let $f:M\to M$, $f(x,y)=(x,y+1)$. The action groupoid $\Z\ltimes_f M$ has no isotropy and the orbits are the columns $\{x\}\times\Z$. Then $\phi:\Z\ltimes_fM\to \Z\ltimes_f M$ given by $\phi(x,y)=(x,0)$ and $\lambda=0$ is a Morita morphism that does not fit our characterization.
\end{example}


\section{The orbit stack of a dynamics}


We review here how to describe maps of stacks as fractions of maps of Lie groupoids, provide some examples, and present our second main theorem. We also derive several corollaries characterizing the orbit stacks of discrete actions and dynamics on connected manifolds. For a thorough discussion on Morita maps, fractions and differentiable stacks we refer to \cite{dh}.

\bigskip


The category of differentiable stacks can be obtained from that of Lie groupoids by localization, namely, formally inverting the Morita maps. The idea is that if $\G \to \H$ is a Morita map, since the groupoids then represent isomorphic stacks, one should be able to find an inverse going in the converse direction $\H \dashto \G$ although it might not be possible to do so with an ordinary Lie groupoid map (e.g. example \ref{ex:2vs3rot} below). Since Morita maps satisfy good saturation properties, every map in the localization can be written as a single fraction.
Given $\G,\H$ Lie groupoids, a {\bf fraction} $\beta/\alpha:\G\dashto\H$ is determined by a Morita morphisms $\alpha:\tilde\G\to\G$ and any other Lie groupoid map $\beta:\tilde\G\to\H$,
$$\G \from \tilde\G \to \H.$$
Two fractions $\beta/\alpha,\beta'/\alpha':\G\dashto\H$ are {\bf equivalent} if there are Morita maps $\gamma,\gamma'$ and isomorphisms $\alpha\gamma\cong\alpha'\gamma'$, $\beta\gamma\cong\beta'\gamma'$.
We then set as a working definition that the category of differentiable stacks and morphisms of stacks is that of Lie groupoids and classes of fractions.
We write $M//\G$ to refer to the {\bf orbit stack} of a Lie groupoid $\G\toto M$.

\begin{remark}\label{rmk:saturated}
It is an instructive exercise to check that a fraction $\beta/\alpha$ is invertible if and only if $\beta$ is also Morita, or in other words, a groupoid morphism is Morita if and only if it induces an isomorphism between the underlying stacks.
\end{remark}

\begin{remark}
Although in this paper we focus on discrete action groupoids, there can be quite different Lie groupoids presenting the same stack. An example arises by considering the {\em suspension} $\R \action \Sigma_{M,f}$ of a discrete dynamics $f\action M$. The associated action groupoid $\R \ltimes  \Sigma_{M,f}$ has 1-dimensional source fibers (instead of the discrete source fibers of $\Z \ltimes M$), but the transverse geometry is the same so that these groupoids are Morita equivalent and they thus define the same orbit stack $M//f \simeq \Sigma_{M,f}// \R$. This flexibility in the description of the orbit stack can be useful in the study of the underlying dynamics; this will be explored elsewhere.
\end{remark}


\begin{example}
Given $M$ a manifold, $G$ a Lie group and $P\to M$ a principal $G$-bundle, we can see it as inducing a fraction $\beta/\alpha:M\dashto G$ where $\alpha:G\ltimes P\to M$ and $\beta:G\ltimes P\to G$ are the projections. This fraction is not invertible unless $G$ is trivial. 
\end{example}

\begin{example}\label{ex:2vs3rot}
Let $f,g:S^1\to S^1$ denote the rotations of period 2 and 3, respectively, defining actions of $\Z_2$ and $\Z_3$ on  the circle. They give rise to Lie groupoids having isomorphic stacks, for there is an obvious fraction of Morita maps $\Z_2\ltimes_f S^1 \from \Z\ltimes\R \to \Z_3\ltimes_g S^1$. In fact the orbit stack of the three groupoids is the manifold $S^1$. Note however that there is not a Morita map relating $\Z_2\ltimes_f S^1$ and $\Z_3\ltimes_g S^1$, as it follows from the Characterization Theorem \ref{thm2}.
\end{example}

The same morphism of stacks $M//\G\to N//\H$ can be represented by several different fractions, and it is natural to seek for special representatives, as we  will do now. Given $\G\toto M$ a Lie groupoid and $\U=\{ U_i \}$ an open cover of $M$, we have a new groupoid 
$$\G_\U=\left(\coprod_{j,i} \G(U_j,U_i)\toto\coprod_{i}U_i\right)$$
where $(y,j)\xfrom{(g,j,i)}(x,i)$ if $x\in U_i$, $y\in U_j$ and $y\xfrom g x$ is an arrow in $\G$. The composition is set by $(h,k,j)(g,j,i)=(hg,k,i)$. The obvious projection $\pi_\U:\G_\U\to\G$ is Morita, and it can be proven (cf. \cite[Prop. 4.5.4]{dh}) that for every fraction $\G\dashto\H$ there is an open cover $\U$ of $M$ and a morphism $\beta$ so that $\beta/\pi_\U$ is equivalent to the original fraction. 

\begin{example}
Coming back to the example of a $G$-principal bundle $P\to M$, viewed as a fraction $\alpha:M\dashto G$, a presentation of it in the form
$M\from \U \to G$
is the same as giving a cocycle $\{\beta_{ji}:U_{ji}\to G\}$ describing $P$.
\end{example}


We present now our second main result, saying that within our framework, every fraction admits a very special representative involving the lifting to the universal covering space.

\begin{theorem}[Representation Theorem]\label{thm2}
Let $G\action M,H\action N$ be actions of discrete groups over connected manifolds. Then every map of stacks $M//G\to N//H$ can be realized by a fraction
$$ (G\ltimes M) \xfrom\pi (\tilde G \ltimes \tilde M) \rightarrow (H\ltimes N),$$
where $\tilde G\action \tilde M$ is the lifted action to the universal cover of $M$ and $\pi$ the quotient map.
\end{theorem}

\begin{proof}
As we have discussed before, the projection map $(\tilde G \ltimes \tilde M) \to (G\ltimes M)$ is a Morita map. We will assume now that $M$ is simply connected and show that every map of stacks $M//G \to N//H$ can be realized by a groupoid map $(G\ltimes M) \to (H\ltimes N)$.
Let us then fix such a map of stacks. We know by the general theory that we can represent it by a fraction $\beta/\pi_\U$ for $\U$ an open cover of $M$,
$$(G\ltimes M) \xfrom{\pi_\U} (G\ltimes M)_{\U} \xto\beta (H\ltimes N).$$
We can assume $U_i$ simply connected for all $i$, and $U_{ji}$ connected for all $j,i$, just by replacing $\U$ with a suitable refinement.
We claim that the following diagram can be completed with an arrow $\tilde\beta$ making the triangle commutative up to isomorphism.
$$\xymatrix{
M_{\U} \ar[r]^\iota & (G\ltimes M)_{\U} \ar[r]^\beta \ar[d]_{\pi_\U} & H\ltimes N \ar[r]^\pi & H\\
& G\ltimes M \ar@{-->}[ur]_{\tilde\beta}
}$$
The upper composition $\pi\beta\iota$ is an $H$-cocycle $c(j,i)$ of the manifold $M$ defined over the cover $\U$. But since $\check H^1(M,\U)=H^1(M)=0$, it has to be a Cech coboundary. Here we are either using non-abelian cohomology, or in a simplified language, the fact that every principal $H$-bundle is a covering space and hence has to be trivial.
One way or the other, we gain a family $d(i)\in H$ such that $c(j,i)=d(j)d(i)^{-1}$.

We can now twist $\beta$ by the $d$ so as to construct a new groupoid map 
$\beta^d: (G\ltimes M)_{\U}\to H\ltimes N$
that on objects and arrows is given by
$$(i,x)\mapsto d(i)\beta(i,x)
\qquad
(j,i,g,x)\mapsto (d(j)\beta(j,i,g,x)d(i)^{-1},d(i)\beta(i,x)).$$
The map $\beta^d$ is isomorphic to $\beta$ and vanishes over the kernel of $\pi_\U$, hence it descends to a map $\tilde\beta$, and the result follows.
\end{proof}

%
%


%

In general, there are many groupoids representing a given stack. Combining our results, we conclude that within our framework there is a distinguished Lie groupoid on each Morita class, namely that on which the manifold of objects is simply connected.


\begin{corollary}\label{cor:1}
Given $G\action M$ and $H\action N$ be actions of discrete groups over connected manifolds, the orbit stacks $M//G$ and $N//H$ are isomorphic if and only if there is an isomorphism $\lambda:\tilde{G} \to \tilde{H}$ between the lifted groups and a $\lambda$-equivariant diffeomorphism $\varphi:\tilde M\to\tilde N$ between the universal covers.
\end{corollary}

\begin{proof}
By Theorem \ref{thm2} we can represent the stack isomorphism $M//G\to N//H$ by a fraction $(G\ltimes M)\from (\tilde G\ltimes\tilde M)\xto\beta (H\ltimes N)$, and according to remark \ref{rmk:saturated}, the map $\beta$ has to be Morita as well. Then, by Theorem \ref{thm1} $\beta$ has to be, up to isomorphism, the lifting of $H\action N$ to the universal cover $\tilde N$. The result now follows.
\end{proof}

In the spirit of corollary \ref{cor:h1}, when working with stacks arising from dynamics $f,g:M \to M$, we can relax the notion of simply connectedness and still have a complete characterization of the dynamics by the stack.

\begin{corollary}
Let $f:M\to M$ and $g:N\to N$ be dynamics over connected manifolds with  $H^1(M)=H^1(N)=0$. The stacks $M//f$ and $N//g$ are isomorphic if and only if there is a diffeomorphism $\varphi:M\to N$ conjugating $f$ to $g$ or to $g^{- 1}$.
\end{corollary}

\begin{proof}
By previous corollary, and since in this case the lifted groups split (cf. example \ref{ex:lifted-group}), there must exists an isomorphism $\lambda:\Z\ltimes_{\tilde f}\pi_1(M)\to\Z\ltimes_{\tilde g}\pi_1(N)$ and an equivariant diffeomorphism $\varphi:\tilde M\cong\tilde N$, where $\tilde f$ and $\tilde g$ are liftings to the universal covers. Writting
$$ \lambda(n,\gamma) = (\lambda_1(n,\gamma), \lambda_2(n,\gamma)) \in \Z\times \pi_1(N), \ \ n\in \Z, \ \gamma \in \pi_1(M),$$
the map $\gamma \mapsto \lambda_1(0,\gamma)\in \Z$ defines an element in $H^1(M)$, which therefore must be trivial. It follows that $\lambda$ maps $ 0 \times \pi_1(M)$ isomorphically to $0 \times \pi_1(N)$ and that $\Z \ni n \mapsto \lambda_1(n,0)\in \Z$ must be an isomorphism, so that $\lambda_1(n,0)=\pm n$. With this, the compatibility between $\varphi$ and $\lambda$ directly implies that $\varphi$ descends to a a diffeomorphism $\bar\varphi:M=\tilde M/\pi_1(M)\to N=\tilde N/\pi_1(N)$ which conjugates $f$ to $g^{\pm 1}$.
\end{proof}

Thus, when $H^1(M)=0$, the geometry of the quotient stack retains faithfully the information about the dynamics modulo conjugation and inversion. For general $M$ non-trivial identifications may take place. Assume $M=N$ for simplicity. 
Picking lifts $\tf,\tg: \tilde M \to \tilde M$ of $f$ and $g$ to the universal cover, a stack isomorphism will be encoded in a pair $\lambda,\varphi$ with
$$ \lambda : \Z\ltimes_{\tilde f} \pi_1(M) \to \Z \ltimes_{\tilde g} \pi_1(M)$$
a group isomorphism satisfying (recall the notation of example \ref{ex:lifted-group}) $$\varphi((n,\gamma)\cdot_{\tf} x) = \lambda(n,\gamma)\cdot_{\tg} \varphi(x)$$
for all $(n,\gamma)\in \Z\times \pi_1(M)$ and $x \in \tilde M$. 
In particular, a necessary condition for $M//f\simeq M//g$ is that there must exist $k\in \Z$ and $\gamma \in \pi_1(M)$ so that
$ \varphi(\tilde f(x)) = (\tilde g^{k} \gamma)(\varphi(x))$ for all $x\in\tilde M$, generalizing the conjugation condition $\varphi f = g^{\pm 1} \varphi$ of the simply connected case. We shall explore this in detail in the case of dynamics on the circle $M=S^1$ in the last section.

\begin{remark}
We sketch here a way to faithfully encode dynamics on manifolds that are not simply connected by orbit stacks defined by extensions.
Given $f:S^1\to S^1$ a dynamics on the circle, we can extend it to $F:S^2\to S^2$ in such a way that the poles are repulsive points and the restriction to the equator is the original $f$. Moreover, this extension can be made canonical. Then, two dynamics $f,g$ on $S^1$ are going to be conjugated if and only if their extensions $F,G$ to $S^2$ are so and, since $S^2$ is simply connected, this holds if and only if the stacks $S^2//F$ and $S^2//G$ are isomorphic. It is tempting to emulate this construction for other non-simply connected manifolds $M$, by realizing them as the boundary of a simply connected manifold $\bar M$, and defining an extension $F$ of $f$ by means of a Morse-Smale flow. We propose this strategy to be explored elsewhere.
\end{remark}

%


\section{Relation to the stacky fundamental group}


Here we provide a conceptual interpretation of our lifted action by developing a down-to-earth approach to stack-theoretic fundamental groups and covering maps. Some of the material presented here is covered in \cite{mm2} in a rather combinatorial spirit, and a topological version of the abstract theory can be found in \cite{n}.

\bigskip


We set $I=[0,1]$ to be the unit interval. 
Given $\G\toto M$ a Lie groupoid and $x\in M$, a {\bf loop} on $M//\G$ based at $x$ is a fraction $\beta/\alpha:I\dashto M//\G$ where, as discussed before, $\alpha$ is a Morita map and $\beta$ is a Lie groupoid map satifying $\beta(g)=u_x$ a unit whenever $\alpha(g)$ is $u_0$ or $u_1$. 
Two such fractions are {\bf equivalent} if there are Morita maps $\gamma,\gamma'$ and isomorphisms $\alpha\gamma\cong\alpha'\gamma'$ and $\beta\gamma\cong\beta'\gamma'$ that are the identity over $0$ and $1$. We say that two loops $\beta/\alpha$, $\beta'/\alpha'$ on $M//G$ based at $x$ are {\bf homotopic} if there is a fraction $\psi/\phi:I\times I\to M//\G$ with $\psi/\phi\mid_{I\times 0}\sim \beta/\alpha$, 
$\psi/\phi\mid_{I\times 1}\sim \beta'/\alpha'$, and $\phi^{-1}(\{0,1\}\times I)\subset\ker\psi$. Loops up to homotopy, with multiplication given by juxtaposition, define the {\bf stacky fundamental group} $\pi_1(M//\G,x)$.


\begin{remark}\label{rmk:path}
By a {\bf regular cover} of the unit interval $I$ we mean a cover $\U=\{U_1,\dots,U_N\}$ in which each $U_i$ is a small connected neighborhood of $[i-1/N,i/N]$. Since these covers are cofinal, every loop in $\G\toto M$ based in $x$ is equivalent to one of the type $\beta/\pi_\U$ with $\U$ a regular cover. Thus a loop can be thought of as a sequence of geometric curves $\gamma_i$ with algebraic jumps $g_j$ linking their endpoints,
$$
\bullet \overset{\gamma_N}{\flowsfrom}
\bullet \overset{g_{N-1}}{\from}
\bullet\cdots
\bullet \overset{\gamma_2}{\flowsfrom}
\bullet \overset{g_1}{\from}
\bullet \overset{\gamma_2}{\flowsfrom}
\bullet
$$
Similarly, homotopies can be generated by elementary moves, resulting from considering product of regular covers $\U\times\U=\{U_i\times U_j\}$. In this way we recover the combinatorial definitions of paths and homotopies in Haefliger's approach to the fundamental group of groupoids (cf. \cite[\S 3.3]{mm2}, see also \cite{mt}).
\end{remark}


Going back now to our groupoids of interest, we can apply our Representation Theorem \ref{thm2} to get the following stacky interpretation of the lifted group of section \ref{sec:lifting}.

\begin{proposition}
Given $G\action M$ an action of a discrete group over a connected manifold, and given $x\in M$, the stacky fundamental group $\pi_1(M//G,x)$ is isomoprhic to the lifted group $\tilde G$ introduced before.
\end{proposition}

\begin{proof}
Given a loop of $M//G$ based at $x$, writing it as a fraction $\beta/\pi_\U$ for a regular cover of $I$,
$$\xymatrix@R=0pt{
& I_\U \ar[dl]_{\pi_\U} \ar[dr]^\beta & \\
I \ar@{-->}[rr]_{\bar\beta}  & & G\ltimes M
}$$
we know from the proof of theorem \ref{thm2} that there is a path $\bar\beta:I\to G\ltimes M$ making the triangle commutative up to a groupoid map isomorphism $\beta\cong\bar\beta\pi_\U$, which can be encoded in a cocycle $\{c_i:U_i\to G\}$. 
Without loss of generality, we can assume $c(1)=1$. 
The loop $\beta/\pi_\U$, with the notations of Remark \ref{rmk:path}, is given by paths $\gamma_i$ linked by groupoid elements $g_j$. It follows that
$c(0)$ is the product of the $g_j$'s while $\bar\beta$ is the path from $c(0)x$ to $x$ obtained by translating the $\gamma_i$ using the $g_j$ to paste these arcs together into a continous path. By lemma \ref{lemma:deck}, in turn, $\bar \beta$ defines a diffeomorphism $\phi_{\bar\beta}$ of $\tilde M$ covering $c(0):M\to M$. We now define $\theta:\pi_1(M//G,x)\to \tilde G\subset G\times\diff(\tilde M)$ by
$\theta(\beta/\pi_\U)=(c(0),\phi_{\bar\beta})$. 
It is straightforward to check that $\theta$ is well-defined, bijective, and preserves the group structure.
\end{proof}


As in topology, the fundamental group classifies covering maps. A map of differentiable stacks $M//\G\to N// \H$ is a {\bf covering} (see eg. \cite{n}) if for every map from a manifold $X\to N// \H$ the stacky fibered product $P=X\times_{N//\H} M//\G$ exists, it is a manifold, and the base-change map $P\to X$ is a covering map of manifolds.
Composition of stacky covering maps and base-change of covering maps are covering maps.
By representing the stacks with Lie groupoids, the stacky fibered product is modeled by the {\em homotopy fiber product} of the Lie groupoids (cf. \cite{dh}).

\begin{example}
A map between manifolds, viewed as differentiable stacks by means of the unit groupoids, is a covering if it is so in classic topology.
\end{example}

\begin{example}
Let $\G\toto M$ be a Lie groupoid, and call $\pi:M\to M//\G$ the canonical projection comming from the units inclusion $M \to \G$. Then $G$ identifies with the fiber product of $\pi$ with itself (cf. \cite{dh}), and the base-change maps are the source and target maps, $G\to M$. Since every map $X\to M/\G$ from a manifold locally factors through the projection $\pi:M\to M//\G$, we conclude that $\pi$ is a covering if and only if $s:\G\to M$ is so.
\end{example}

A particular case of previous example is that of an action groupoid of an action of a discrete group. We say that a connected stacky covering map $N//\H\to M//\G$ is {\em universal} if any other connected cover $N'//\H'\to M//\G$ is a quotient of it.
As in the manifold case, the group of deck transformations provides us with a model for the stacky fundamental group that is independent of the base point. Moreover, the following proposition allows us to interpret the lifted action of $\tilde G$ on $\tilde M$ as stacky deck transformations.

\begin{proposition}\label{prop:fundG}
Given $G\action M$ an action of a discrete group on a connected manifold,
the composition $q:\tilde M\to M\to M//G$ is the stacky universal cover, and the group
$$\deck(\tilde M\to M//G)=\{(\phi,\alpha):\phi:\tilde M\to\tilde M,\alpha:q\phi\then q\}$$
is isomorphic to the stacky fundamental group $\pi_1(M//G,x) \simeq \tilde G$.
\end{proposition}

\begin{proof}
Note that $\tilde M\to M\to M//G$ is a composition of covering maps. Given any other connected covering map $N//\H\to M//G$, in the stacky fiber product
$$\xymatrix{
P \ar[r] \ar[d] & N//\H \ar[d] \\ \tilde M \ar[r] \ar@{-->}[ur] & \G
}$$
the map $P\to \tilde M$ has to be a trivial covering, then it admits a section, which composed with the projection $P\to N//\H$ gives the diagonal covering.

Regarding the second statement, 
an isomorphism of groupoid morphisms
$$\xymatrix@C=0pt{
(\tilde M\toto\tilde M) \ar@{}[drr]|{\alpha\Leftarrow}\ar[dr]_q \ar[rr]^\phi & & (\tilde M\toto\tilde M) \ar[dl]^q\\
& (G\ltimes M\toto M) & }$$
is given by a map $\alpha:\tilde M\to G\times M$, whose image must be included in a component $g\times M$, and satisfy that $s\alpha=q:\tilde M\to M$ and $t\alpha=q\phi$, from where $\rho_gq=q\phi$. The bijection $\deck(\tilde M\to M//G)\cong\tilde G$ with the lifted group becomes evident, and it is easy to check that it preserves the group composition (see \cite{mm2,n}).
\end{proof}

Our final general statement of this section is that the stacks arising from actions $G\action M$ of discrete groups over connected manifolds are exactly the differentiable stacks admiting a manifold as a covering.

\begin{proposition} \label{prop:stcover}
If $M//\G$ is a differentiable stack and $q:N\to M//\G$ is a stacky covering map from a connected manifold, then $M//\G$ is isomorphic to the orbit stack of a discrete group $H$ acting on $N$.
\end{proposition}

\begin{proof}
By replacing $N$ with its universal cover, we may assume that $N$ is simply connected. Then the stacky fiber product of $q$ with itself
$$\xymatrix{
P \ar[r] \ar[d] & N \ar[d] \\  N \ar[r] & M//\G
}$$
give us a manifold $P$, and considering the iterated fiber products, we gain a Lie groupoid structure on $P\toto N$ such that $N//P\cong M//\G$. Since the source and target are base-change of $q$, they are covering maps, and since $N$ is simply connected, we get a trivialization of $s$ as $P=N\times H\to N$, with $H$ discrete. Finally, $H$ inherits a group structure by taking connected components on the groupoid, $\pi_0(P)=H\toto \pi_0(N)=\ast$, and the target $H\times N\to N$ gives the desired action.
\end{proof}

\section{The characteristic class of a dynamics}


In this section we show how to completely characterize a dynamics from a stacky viewpoint. We do it by identifying a cohomology class $c_f\in H^1(M//f)$ that allow us to recover $M$ and $f$ up to conjugation. We call $c_f$ the {\em characteristic class} of $f$. We explain how to reconstruct the whole data from $c_f$, and explain the analogy with the classical theory of characteristic classes of principal bundles.

\bigskip


Given a dynamics $f:M \to M$, recalling that the acting group is $G =\Z$, we have that $\tilde G_f \subset \Z \times \diff(\tilde M)$ consists of all lifts of $f$ and its iterates to the universal cover $\tilde M$ (see Lemma \ref{lemma:deck} and Example \ref{ex:lifted-group}). The first projection then defines a map $\tilde G_f \to \Z$ and, using the isomorphism of Proposition \ref{prop:fundG}, we get a group morphism
$$ c_f: \pi_1(M//f,x) \simeq \tilde G_f \to \Z\subset\R $$
This $c_f$ is characterized by mapping the lifts of $f$ to 1 and the geometric loops $\gamma \in \pi_1(M,x)\subset \pi_1(M//f,x)$ to 0. In particular, it does not depend on the base point $x$. 


This $c_f$ has a clean interpretation in terms of cohomology of stacks. Given $\G\toto M$ a Lie groupoid, its {\bf Bott-Shulman complex} $(\Omega^q(\G_p),d_h,d_v)$ is the double complex of differential forms over the {\em nerve} of $\G$, whose horizontal and vertical differentials are given by the simplicial structure and the de Rham differential, respectively. The  cohomology $H(\Omega^\bullet(\G_\bullet))$ of the total complex is a Morita invariant, and is called the {\bf cohomology} of the stack $M//\G$. Details can be found in \cite{be}.


We are interested here in the very particular case of $H^1$.  Before establishing the main theorem, we will relate the stacky fundamental group with the stacky cohomology. A key intermediate ingredient here is the stacky version of singular homology $H^{sing}_\bullet(M//\G)$, for which we refer the reader to \cite{be}. 

\begin{lemma}
Given $\G\toto M$ a Lie groupoid with $M$ connected, there are canonical isomorphisms: 
$$H^1(M//\G)
\cong {\rm Hom}(H^{sing}_1(M//\G),\R)
\cong {\rm Hom}(\pi_1(M//\G,x),\R).$$
\end{lemma}

\begin{proof}
The first isomorphism is given in \cite[\S \emph{Relation to de Rham cohomology}]{be}, it comes from the pairing between de Rham cohomology and singular homology given by integrating forms over cycles, and works for stacks analogously as it does for manifolds.
The second isomorphisms is a direct consequence of a stacky version of Hurewicz Theorem, namely, an isomorphism between the first homology group and the abelianization of the fundamental group. The proof of this second statement can be carried out following the classical arguments, as we now sketch. We write $C\subset \pi_1(M//\G,[x])$ for the commutator subgroup, we define group homomorphisms $h$ and $k$ at the level of loops and cycles below,
$$
\pi_1(M//G,[x])/C\xto h H^{sing}_1(M//G),
\qquad
H^{sing}_1(M//G)\xto k \pi_1(M//G,[x])/C,
$$
and argue that they are mutual inverses.
For the sake of simplicity, we use the combinatorial approach to the stacky fundamental group, discussed in \ref{rmk:path}, in which a loop based at $x$ is given by a sequence 
$$(\gamma_N,g_{N-1},\dots,\gamma_2,g_1,\gamma_1)$$
of paths $\gamma_i$ in $M$ and arrows $g_j$ in $\G$ such that $s(g_i)=\gamma_i(1)$, $\gamma_{i}(0)=t(g_{i-1})$ and $\gamma_N(1)=\gamma_1(0)=x$. 
On the other hand, following \cite[\S \emph{Singular homology}]{be}, stacky singular 1-cycles admit a similar description as a formal sum of paths and arrows with matching boundaries (see \cite[\S \emph{What do cycles look like?}]{be} for details). We can then define $h$ by mapping a stacky loop as above to the 1-cycle $[\sum \gamma_i + \sum g_i]$. The inverse $k$ can be defined at the level of basic 1-chains, $y'\overset\gamma\flowsfrom y$ and $y'\xfrom g y$, by fixing smooth paths $\gamma_y:x\leadsto y$ for every $y\in M$, and setting $k(\gamma)=[\gamma_{y'}^{-1}\ast\gamma\ast\gamma_y]$ and $k(g)=[(\gamma_{y'}^{-1},\gamma,\gamma_y)]$. It is straightforward to check that $h$ and $k$ are indeed well-defined, group morphisms, and mutual inverses.
\end{proof}

In our case of a dynamics $f:M\to M$ over a manifold, we can regard $c_f$ as a stacky 1-cocycle associated to $f$, we call it the {\bf characteristic class} of $f$.


\begin{theorem}
Given $f:M\to M$ a dynamics on a connected manifold,
the stack $M//f$ and the stacky 1-cocycle $c_f\in H^1(M//f)$ encode $f$ up to conjugation.
\end{theorem}

\begin{proof}
If $(\lambda,\phi): \tilde G_{f_1} \ltimes \tilde M_1 \to \tilde G_{f_2} \ltimes \tilde M_2$ is a groupoid map inducing an isomorphism of stacks $M_1//f_1 \simeq M_2//f_2$ (c.f. Corollary \ref{cor:1}), and moreover $(\lambda,\phi)$ preserves the associated stacky 1-cocycles, then we can write $c_{f_1} = c_{f_2} \circ \lambda$ and it is easy to see that $\phi:\tilde M_1\to \tilde M_2$ must descend to a map $M_1\to M_2$ that gives a conjugation $f_1 \sim f_2$. Conversely, any conjugation $f_1\sim f_2$ can be lifted to a groupoid isomorphism preserving the stacky 1-cocycles as above.  
\end{proof}


Let us discuss the reconstruction of a dynamics out of the pair $(\X,c)$, where $\X$ is a differentiable stack and $c \in H^1(\X)$ is a stacky 1-cocycle satisfying the following conditions. First, $\X$ admits a covering by a manifold (c.f. Proposition \ref{prop:stcover}). Let us denote $\tilde M \to \X$ the universal cover, in which $\tilde M$ is thus a simply connected manifold such that $\X\simeq \tilde M//\tilde G$ for some discrete group action $\tilde G \action \tilde M$. Second, choosing a reference point $x\in \tilde M$, the representation of the cocycle as a group morphism $c : \pi_1(\tilde M//\tilde G,x) \to \Z$ is a surjective map and its kernel $K=Ker(c)$ acts freely and properly on $\tilde M$ by restriction of the action of $\pi_1(\tilde M//\tilde G,x)\simeq \tilde G$ on $\tilde M$ to $K$ (c.f. Proposition \ref{prop:fundG}). We then get the manifold $M$ as the quotient of $\tilde M$ by $K$ together with a residual action of $\Z = \pi_1(\tilde M//\tilde G, x)/ K$ on $M$, so that the \emph{induced dynamics} $f:M\to M$ is then given by the action of $1\in \Z$ on $M$. Notice that the choice of a loop  $l\in \pi_1(\tilde M//\tilde G, x)$ such that $c(l)=1$ is equivalent to the choice of a lift $\tilde f: \tilde M \to \tilde M$ of $f$ to the universal cover.


\begin{remark}
We close by explaining the analogy with the classical theory of characteristic classes. To any principal $G$-bundle $E\to M$ we can associate its Gauss map $g:M\to BG$  into the classifying space of $G$, so $E$ is isomorphic to the pullback of the universal principal $G$-bundle by $g$. Then the image of $g^*:H^\bullet(BG)\to H^\bullet(M)$ is the subalgebra of characteristic classes of $E$ within $H^\bullet (M)$. When $H^\bullet(BG)$ is canonically a free algebra, the characteristic classes of $E$ can be encoded by the images $c_i(E)\in H^\bullet(M)$ of the generators along $g^*$.
In our case, given a dynamics $f:M\to M$, the quotient map $M\to M//f$ is not only a stacky covering map, but moreover a stacky principal $\Z$-bundle (cf. \cite{dh}). This is the pullback of the universal one by the map $g:M//f\to \ast//\Z$, which at the level of groupoids is just the projection 
$(\Z\ltimes_fM\toto M)\to (\Z\toto \ast)$. A simple computation shows there are canonical isomorphism $H^q(\ast//\Z)\cong\R$ for $q=0,1$ and $H^q(\ast//\Z)=0$ otherwise. Then $c_f=g^*(c)$ is the pullback of the canonical generator of $H^1(\ast//\Z)$. 
\end{remark}



\section{Applications and examples}


In this section we apply our results in simple examples of dynamics, and relate the outcome to some results appearing in the literature. 

\subsection*{Rotations on the circle}

Let $M=S^1$, with universal cover $\tilde M=\R$ and $\pi_1(M)=\Z$ acting by integer translations. We shall consider dynamics given by rigid rotations $r:S^1 \to S^1$. Such a rotation can be uniquely lifted to a rigid translation 
$$ x \mapsto x + \tau, x \in \R$$
with $\tau \in [0,1)$ called the \emph{rotation number} of $r$.
In this case, we have that the lifted group is $\tilde G = \Z \times \Z$ and we denote the induced lifted action by 
$$(\Z\times\Z)\action_{\tau,1}\R, \ \ (m,n) \cdot x = x + m\tau + n.$$

When the rotation number $\tau=p/q$ is \emph{rational}, we have that the coarse orbit space $S^1/r$ inherits a smooth structure diffeomorphic to an ordinary circle $S^1/r \simeq S^1$ and such that the quotient map defines a covering $S^1 \to S^1/r$. At the level of stacks, considering $p$ and $q$ coprime so that there exist $a,b \in \Z$ satisfying $aq - bp = 1$, we get that $\lambda: \Z \times \Z \to \Z \times \Z, (m,n) \mapsto (a m + b n,pm +q n)$ is an isomorphism which together with $\phi:\R \to \R, \ x \mapsto q x$ define an isomorphism of actions
$$(\phi,\lambda): (\Z\times\Z)\action_{p/q,1}\R \simeq (\Z\times\Z)\action_{0,1}\R.$$
On the right hand side above, one of the $\Z$-factors acts trivially and the other acts by integer translations. The quotient stack is thus isomorphic to a product
$$ S^1//r \simeq S^1 \times \ast//\Z $$
where the manifold $S^1$ is seen as a stack and $\ast//\Z$ denotes the quotient stack associated to the trivial $\Z$-action over a 1-point space $\ast$. 
Then, the orbit spaces of $S^1$ by rational rotations are all isomorphic both as coarse orbit spaces or as orbit stacks and the orbit stack contains essentially the same information as the coarse orbit space.

Let us now consider an \emph{irrational rotation}, namely, a rigid rotation $r$ of $S^1$ such that its rotation number $\tau$ is irrational. In this case, the orbit through any point $x\in S^1$ is dense and thus the coarse orbit space $S^1/r$ inherits the trivial topology. The quotient stack $S^1//r$, though, is able to retain non-trivial information of the original dynamics:

\begin{proposition}
Let $r_1,r_2$ be two rigid rotations acting on $S^1$ with irrational rotation numbers $\tau_1,\tau_2$, respectively. Then, the two associated quotient stacks are isomorphic, $S^1//r_1 \simeq S^1//r_2$, iff $\tau_1,\tau_2$ are in the same orbit of the $GL_2(\Z)$-action by homographies:
$$ \tau \mapsto \frac{a \tau + b}{c \tau + d}, \ \left(\begin{array}{cc}a & c \\ b & d \end{array}\right)\in GL_2(\Z).$$
\end{proposition}

\begin{proof}
By Corollary \ref{cor:1}, $S^1//r_1 \simeq S^1//r_2$ if and only if there exists an isomorphism of the lifted actions 
$$(\varphi,\lambda): (\Z\times\Z)\action_{\tau_1,1}\R \simeq (\Z\times\Z)\action_{\tau_2,1}\R.$$
Such an isomorphism consists of a group automorphism $\lambda = \left(\begin{array}{cc}a & c \\ b & d \end{array}\right) \in Gl_2(\Z)$ and a diffeomorphism $\varphi: \R \to \R$ such that for all $x\in \R$, $m,n\in \Z$ we have:
$$\varphi(x +  m\tau_1 + n) = \varphi(x) + (am+cn) \tau_2 + (bm + dn)$$

First, suppose that  $\tau_1 = \frac{a \tau_2 + b}{c \tau_2 + d}$, where $\lambda = \left(\begin{array}{cc}a & c \\ b & d \end{array}\right) \in Gl_2(\Z)$. Then we have an isomoprhism $(\varphi,\lambda)$ of actions as above, by defining $\varphi(x)= (c\tau_2 + d)x$.

Conversely, assume we have such an isomorphism of actions $(\varphi,\lambda)$. Let $m_i,n_i\in \mathbb N$ be such that $\underset{i\to \infty}{\lim}  m_i \tau_1 + n_i = 0$.
We then have by continuity that
$$\varphi(0) = \underset{i\to \infty}{\lim} \varphi(0+ m_i \tau_1 + n_i) = 
  \underset{i\to \infty}{\lim} [\varphi(0) + (a \tau_2+b) m_i + (c \tau_2 +d)n_i] $$
from where $\frac{a\tau_2 + b}{c \tau_2 + d} = \tau_1$ as wanted.
\end{proof}

The above result can be seen as a geometric analogue of a result by Rieffel in the context of $C^*$-algebras \cite{r}. In fact, the construction $\G\mapsto C^*(\G)$ that associates to a Lie groupoid its convolution algebra is functorial, in the sense that it transforms Morita equivalences of Lie groupoids into Morita equivalences of $C^*$-algebras (see eg. \cite{l}), and one of the $C^\ast$ implications of our result is already in \cite{con}.

For general diffeomorphisms of the circle, since conjugated dynamics yield isomorphic orbit stacks, the above proposition generalizes to any pair diffeomorphisms of $S^1$ which are conjugated to rotations. We believe it is further possible to provide a characerization of rotation numbers of general $f\action S^1$ in terms of the underlying quotient stack; these ideas will be explored elsewhere.


\subsection*{Hyperbolic toral automorphisms}


Let $X=T^n=\R^n/\Z^n$ be the $n$-torus, viewed as the quotient of the Euclidean space by the integer lattice. We shall consider hyperbolic toral automorphisms, namely maps $f_A:X\to X$ given by a matrix $A\in Gl_n(\Z)$ with no eigenvalues of norm 1. These are the simplest examples of Anosov diffeomorphisms. 


Given such a dynamics $f_A:X\to X$, the corresponding orbit topological space $X/f_A$ is very hard to handle. While the rational points are periodic, hence inducing closed points on the quotient, the other points have dense orbits, and any non-trivial open in the quotient must contain their images. Next we will show that the orbit stack allows us to classify them up to conjugacy and inversion.


Given $f_A:X\to X$, the lifted dynamics to the universal cover is 
the action by affine transformations $(\Z\ltimes_A\Z^n)\action \R^n$,  $(n,v)\cdot x=A^nx+v$. 
As it follows from Corollary \ref{cor:1}, the orbit stack keeps track of the action of the dynamics on the fundamental group. In this case, this is enough to recover the dynamics up to conjugation and inversion.

\begin{proposition}
Two dynamics given by hyperbolic toral automorphisms arising from the matrices $A,B$ yield isomorphic orbit stacks if and only if $A$ is conjugated to $B^{\pm 1}$.
\end{proposition}

\begin{proof}
Notice that for any $A$ we have isomorphisms $\tilde G_A\cong \Z\ltimes_A\Z^n$, where the semi-direct product structure is the affine one, $(n,v)(n',v')=(n+n',v+A^nv')$.
If the stacks $T^n//f_A$ and $T^n//f_B$ are isomorphic, then the lifted groups to the universal cover $\tilde G_A$ and $\tilde G_B$ have to be isomorphic as well. Let us study the group $\tilde G_A$ in more detail.
Direct computation shows that the inverse of $(n,v)$ is $(-n,-A^{-n}v)$, and that the commutator of two elements is
$$[(n,v),(n',v')]=(0,v+A^nv'-A^{n'}v-v')$$
which shows that $[\tilde G_A,\tilde G_A]\subset0\times\Z^n$. Conversely, 
since $[(n,v),(1,0)]=(0,(I-A)v)$ and $1$ is not an eigenvalue of $A$, we conclude that $[\tilde G_A,\tilde G_A]=0\times\Z^n$.
It follows that the isomorphism $\lambda$ is a matrix block as follows,
$$\lambda:\begin{bmatrix}\pm 1 & 0\\\ast & \lambda'\end{bmatrix}$$
and $\lambda(1,0)=(\pm1,v_0)$. Lastly, if we denote by $c_{(n,v)}$ the operator which conjugates by $(n,v)$, we note that $c_{(n,v)}(0,v')=(0,A^nv')$, and from
$$(0,\lambda'Av)=\lambda(0,Av)=\lambda c_{(1,0)}(0,v)=c_{(\pm1,v_0)}(0,\lambda'v)=(0,B^{\pm1}\lambda'v)$$
we conclude that $A$ and $B^{\pm1}$ are conjugated by $\lambda'$.
\end{proof}

We look forward to investigate Anosov diffeomorphism by using stack techniques, building over the new theory of metrics on differentiable stack, currently in development \cite{dhf2}.


\subsection*{Periodic dynamics and orbifolds}

We consider now dynamics $f:M\to M$ that are periodic. More generally, let $G\action M$ be an action of a discrete group over a connected manifold such that the induced morphism $\rho:G\to \diff(M)$ has finite image. Writing $K$ for the kernel of $\rho$, we have that $K\action M$ is trivial, and we gain a short exact sequence of Lie groupoids.
$$1\to K\ltimes M \to G\ltimes M \to G/K\ltimes M \to 1$$

The quotient stack of the reduced action $G/K\ltimes M$ is an {\bf orbifold}. As any orbifold, it can be locally modeled by a linear action of a finite group on an Euclidean space. In fact, around $x\in M$, we can locally define a Morita equivalence between $G/K\ltimes M$ and the action groupoid of the tangent representation $(G/K)_x\ltimes T_xM$.
This follows for instance from the linearization theorem for proper Lie groupoids (see eg. \cite{dh} and references therein).

\begin{example}
Given $M=S^2$ and $f_p:S^2\to S^2$ a rotation of period $p$ around the $z$ axis, all the coarse orbit spaces $S^2/f_p$ are homeomorphic to $S^2$, but the quotient stacks $S^2//f_p$ are all different, for if $(S^2//f_p)'$ is the corresponding orbifold, then its isotropy at the poles is exactly $\Z_p$. The isotropy at the poles keeps track of the size of the orbits near them.
\end{example}

The orbifolds arising this way are {\em good orbifolds} in the sense of Thurston \cite{t}: their universal cover is a manifold. A paradigmatic example of a bad orbifold, that cannot be obtained as the stack underlying an action of a discrete group on a connected manifold, is the {\em teardrop}, whose southern hemisphere equals that of the sphere $S^2$, but whose northern hemisphere is a cone modeled by the action $\Z_2\action D^2$, $\bar 1\cdot x=-x$.

In the next subsection we will see that it may happen that two essentially finite actions may have isomorphic orbifolds and yet not be isomorphic as stacks. In some sense, what is happening is that the sequence of stacks
$$M//K \to M// G \to M//(G/K)$$
can be a non-trivial extension or fibration.

\subsection*{Construction of lens spaces}

\def\C{{\mathbb C}}

This example is related to the theory of lens spaces $L(p,q)$, which are quotients of of the three-sphere $S^3\subset\C^2$. Given $p,q$ coprime integers, we write $a=\exp(2\pi i/p)$ and consider the diffeomorphism
$$f_{p,q}:S^3\to S^3 \qquad f(z,w)=(az,a^qw)$$
Then $f_{p,q}$ yields a $\Z$-action which has no fixed points and has period $p$. The orbit space $L(p,q)=S^3/f_{p,q}$ is called a lens space, it is a manifold with fundamental group $\Z_p$, and the quotient map $S^3\to L(p,q)$ is the universal cover.

Given a map $f$ between lens spaces, its {\bf degree} 
$\deg f$ is defined as the degree of a lift $\tilde f:S^3\to S^3$. This is well-defined up to the actions and, therefore, it is well-defined up to homotopy, for each action is by maps isotopic to the identity.

Deep results on algebraic topology, using simple homotopy theory or other sofisticated tools, show the following facts:

\begin{itemize}
\item $L(p,q)$ and $L(p,q')$ are homotopy equivalent if and only if $qq'=\pm x^2$ has solutions mod $p$;
\item $L(p,q)$ and $L(p,q')$ are homeomorphic if and only if $q'=\pm q^{\pm 1}$ mod $p$;
\end{itemize}
Moreover, when $L(p,q)$ and $L(p,q')$ are homeomorphic, we have very explicit homeomorphisms relating them:
$$\tau:L(p,q)\to L(p,q^{-1}) \quad \tau[z,w]=[w,z]$$
$$\sigma:L(p,q)\to L(p,-q) \quad \sigma[z,w]=[z,\bar w]$$
from where it follows that they are also diffeomorphic, and moreover isometric for the obvious Riemannian metric. 

Next we will show that the orbit stacks can catch even more information than the isotropy groups and the orbit space.

\begin{proposition}
The orbit stacks $S^3//f_{p,q} \to S^3//f_{p,q'}$ are isomorphic if and only if $q'=\pm q$.
\end{proposition}

\begin{proof}
Since $S^3$ is simply connected, in light of our main theorems, the stacks are isomorphic if and only if $f_{p,q}$ is conjugated to $f_{p,q'}^ {\pm 1}$. We can directly check that $f_{p,q}$ and $f_{p,-q}$ are conjugated by $\sigma$. For the converse, since isomoprhic stacks must have homeomorphic orbit spaces, we can already assume that $q'=\pm q^ {\pm 1}$ mod $p$. 
Assuming the existence of an isomorphism of stacks $S^3//f_{p,q} \to S^3//f_{p,q^{-1}}$, we will show that $q^{-1}=\pm q$, concluding the proof. 

The diffeomorphisms $f_{p,q}$ and $f_{p,q^{-1}}^q$ are conjugated by $\tau$. From the stack isomorphism we have that $f_{p,q^{-1}}$ and $f_{p,q}^{\pm 1}$ are conjugated. By transitivity, it follows that $f_{p,q}$ must be conjugated to $f_{p,q}^{\pm q}$. If $\theta:S^3\to S^3$ is such a conjugation, then $\theta$ must have degree $\pm 1$ since it is a diffeomorphism, and $\theta_*:\pi_1(L)\to\pi_1(L)$ must be multiplication by $q$. It follows from the theory of torsion that $q^2=\pm 1$ mod $p$, and therefore $q=\pm q^ {-1}$ mod $p$. 
\end{proof}

\begin{example}
We consider the six maps $f_{7,q}$, $1\leq q<7$.
All of them lead to the same quotient homotopy type, namely $L(7,1)$, $L(7,2)$, $L(7,3)$, $L(7,4)$, $L(7,5)$ and $L(7,6)$ are homotopy equivalent. If we look at the quotient manifolds, then we find two homeomorphism and diffeomorphism types, namely that of $L(7,1)$ and $L(7,6)$, and that of $L(7,2),L(7,3),L(7,4)$ and $L(7,5)$. If, instead of looking at the quotient as a manifold, we look at the quotient stacks $\mathbb{L}(7,q)=S^3//f_{7,q}$, then we find three different stacks since  $\mathbb{L}(7,3)\simeq \mathbb{L}(7,4) \not\simeq \mathbb{L}(7,2)\simeq \mathbb{L}(7,5)$ by the previous proposition.
\end{example}

We close by rising the question of whether it is possible to achieve the classification of lens spaces by using stack techniques, and avoiding simple homotopy theory or other sophisticated tools.

\end{document}